\documentclass{amsart}
\usepackage{xypic}
\usepackage{graphicx, amsmath, amssymb, amsthm}
\usepackage{tikz-cd}
\usepackage{hyperref} 

\newcommand{\NSFThree}{NSF Grant DMS-1509652}

\newcommand{\Z}{{\mathbb  Z}}

\DeclareMathOperator{\Hom}{Hom}

\newcommand{\res}{res}
\newcommand{\tr}{tr}

\newcommand{\m}[1]{{\protect\underline{#1}}}

\newcommand{\mM}{\m{M}}

\newcommand{\mR}{\m{R}}
\newcommand{\mB}{\m{B}}

\newcommand{\mA}{\m{A}}

\newcommand{\cc}[1]{\mathcal #1}

\newcommand{\cO}{\cc{O}}

\newcommand{\cP}{\cc{P}}

\newcommand{\Set}{\mathcal Set}

\newcommand{\Ninfty}{N_\infty}

\newcommand{\Tamb}{\mathcal Tamb}
\newcommand{\OTamb}{\cO\mhyphen\Tamb}

\newcommand{\OpRes}{i_{\cO}^{\cO'}\!}

\mathchardef\mhyphen=45

\newtheorem{theorem}{Theorem}[section]
\newtheorem{lemma}[theorem]{Lemma}
\newtheorem{corollary}[theorem]{Corollary}

\newtheorem{definition}[theorem]{Definition}
\newtheorem{proposition}[theorem]{Proposition}

\newtheorem{remark}[theorem]{Remark}

\newtheorem{notation}[theorem]{Notation}

\newcommand{\defemph}[1]{\textbf{#1}}

\begin{document}

\title[Tambara Right Adjoint]{The right adjoint to the equivariant operadic forgetful functor on incomplete Tambara functors}

\author[A.~J.~Blumberg]{Andrew~J. Blumberg}
\address{University of Texas \\ Austin, TX 78712}
\email{blumberg@math.utexas.edu}
\thanks{A.~J.~Blumberg was supported in part by NSF grant DMS-1151577}

\author[M.~A.~Hill]{Michael~A. Hill}
\address{University of California Los Angeles\\ Los Angeles, CA 90095}
\email{mikehill@math.ucla.edu}
\thanks{M.~A.~Hill was supported in part by {\NSFThree}}

\begin{abstract}
For $N_\infty$ operads $\cO$ and $\cO'$ such that there is an
inclusion of the associated indexing systems, there is a forgetful
functor from incomplete Tambara functors over $\cO'$ to incomplete
Tambara functors over $\cO$.  Roughly speaking, this functor forgets
the norms in $\cO'$ that are not present in $\cO$.  The forgetful
functor has both a left and a right adjoint; the left adjoint is an
operadic tensor product, but the right adjoint is more mysterious.  We
explicitly compute the right adjoint for finite cyclic groups of prime
order.
\end{abstract}

\maketitle

The complexity of equivariant stable homotopy theory is encoded in the
structure of the additive and multiplicative transfers that connect
the stable homotopy groups $\pi_*\big((-)^H\big)$ as the subgroup $H \subset
G$ varies.  The compatibilities between transfers can be encoded
operadically, using equivariant generalizations of the classical
$E_\infty$ operad, referred to as $\Ninfty$ operads~\cite{BHNinfty}.
At one extreme is the non-equivariant $E_\infty$ operad regarded as a
trivial $G$-operad; algebras have a coherently
commutative multiplication.  At the other is a ``genuine''
$G$-$E_\infty$ operad; algebras have transfers for all pairs of
subgroups.   In the category of based $G$-spaces, $N_\infty$ operads
structure different kinds of infinite loop spaces (or equivalently
connective $G$-spectra); the operadic transfers are the usual stable transfers.
In the category of $G$-spectra, $N_\infty$ operads structure variants
of $E_\infty$ ring spectra; the operadic transfers are the
Hill-Hopkins-Ravenel multiplicative norms.  

The operadic transfers of an $N_\infty$ algebra give rise to
interesting algebraic structures; we focus on the case of $N_\infty$
algebras in spectra.  Homotopy groups of genuine equivariant spectra
assemble into Mackey functors, and there is a symmetric monoidal
abelian category of Mackey functors. The zeroth homotopy group of an
algebra over a trivial $E_\infty$ operad is a commutative monoid
object in Mackey functors: a Green functor. In contrast, Brun showed
that the zeroth homotopy group of an algebra over a $G$-$E_\infty$
operad is a Tambara functor~\cite{Brun}, which loosely speaking is a
Green functor together with multiplicative transfers. Building on this
work, we defined for any $\Ninfty$ operad $\cO$ an $\cO$-Tambara
functor, showing that the zeroth homotopy group of an $\cO$-algebra is
naturally an $\cO$-Tambara functor~\cite{BHOTamb}. Conceptually, an
$\cO$-Tambara functor is a Green functor together with the norms
parameterized by $\cO$.

One of the most interesting aspects of the classification of $\Ninfty$
operads is that it is essentially algebraic.  Specifically, the
homotopy category of $\Ninfty$ operads is very simple: it is a finite
poset \cite[Theorem 3.24]{BHNinfty}, \cite{Rubin}, \cite[Corollary IV]{BonventrePer}, \cite[Section 4]{GutierrezWhite}.  The point is that
$N_\infty$ operads are completely described by algebraic data
(``indexing systems'') which records exactly which norms are present.
As a consequence of the study of $\cO$-Tambara functors, we have an
equivalent reformulation which makes this poset easier to digest.

\begin{proposition}[{\cite[Theorem 3.18]{BHOTamb}}]
The homotopy category of $\Ninfty$ operads is equivalent to the poset
of pullback stable, finite coproduct complete, wide (containing all
the objects) subcategories of the category $\Set^G$ of finite
$G$-sets. 
\end{proposition}

Because of this equivalence, we find it convenient to elide the
distinction between the operad and the subcategory of $\Set^G$. In
particular, we abuse notation by using symbols like inclusions for
operads and conversely referring to subcategories as ``operads''.

As we vary the operad, there are evident relationships between the
categories of algebras.  The construction of the forgetful functor is
clear.

\begin{definition}
If $\cO\subset\cO'$ is a pair of $\Ninfty$ operads, then there is an operadic forgetful functor
\[
{\OpRes}\colon\cO'\mhyphen\Tamb\to\OTamb
\]
which ``forgets the norms in $\cO'$ which are not in $\cO$''.
\end{definition}

Moreover, this forgetful functor participates in two adjunctions.

\begin{proposition}[{\cite[Proposition 5.16]{BHOTamb}}]
The forgetful functor ${\OpRes}$ has both left and right
adjoints:
\begin{center}
\begin{tikzcd}
\cO'\mhyphen\Tamb
	\arrow[r, "{{\OpRes}}" description]
&
\OTamb
	\arrow[l, "{\cO'\otimes_{\cO}(-)}"', bend right]
	\arrow[l, "{F_{\cO}(\cO',-)}", bend left]
\end{tikzcd}
\end{center}
\end{proposition}

The left adjoint $\cO'\otimes_{\cO}(-)$ is a familiar construction; it
can be described as an operadic tensor product.  Moreover, it is
conceptually simple: we formally put in any of the norms we are
missing.  The right adjoint is much stranger.  

Since Tambara functors are not yet a familiar object to many
topologists, our goal in this paper is to illustrate several aspects
of computation with them by working out explicit formulas describing
the right adjoint to the operadic forgetful functor ${\OpRes}$
from $\cO'$-Tambara functors to $\cO$-Tambara functors for $G=C_2$.

The strategy of analysis is formal: there are enough ``free''
functors, and this gives a canonical co-$\cO$-Tambara functor object
in $\cO$-Tambara functors. Maps out of this co-$\cO$-Tambara functor
to an $\cO$-Tambara functor $\m{R}$ then canonically recovers $\m{R}$
as an $\cO$-Tambara functor.

We begin in Section~\ref{sec:crash} with a quick review of incomplete
Tambara functors.  In Section~\ref{sec:freeab} we study the free
incomplete Tambara functor for any group $G$; we specialize to the
case $G = C_2$ in Section~\ref{sec:freeC2}.  Finally, we give explicit
formulas for the operadic right adjoint to the forgetful functor in
Section~\ref{sec:main}.

\section{A crash course in \texorpdfstring{$\cO$}{O}-Tambara functors}\label{sec:crash}

In this section, we quickly review the algebraic characterization of
incomplete Tambara functors we introduced in~\cite{BHOTamb}.  Let
$\cO$ be a wide, pullback stable, finite coproduct complete
subcategory of $\Set^G$.  We can associate to such a category a new
category of ``polynomials with exponents in $\cO$''.  

\begin{definition}\label{defn:Polynomials}
Let $\cP^G_{\cO}$ denote the category of polynomials with exponents in $\cO$ in $G$-sets. The objects are finite $G$-sets, and the morphisms are isomorphism classes of diagrams
\[
S\xleftarrow{f} U\xrightarrow{g} V\xrightarrow{h} T,
\]
where $g\in\cO$ and where two such diagrams are isomorphic if we have a commutative diagram
\[
\xymatrix@R=.1\baselineskip{
{} & {U'}\ar[r]^{g'}\ar[dd]_{\cong}\ar[ld]_{f'} & {V'}\ar[dd]^{\cong}\ar[rd]^{h'} & {} \\
{S} & {} & {} & {T.} \\
{} & {U}\ar[r]_{g}\ar[lu]^{f} & {V}\ar[ru]_{h} & {}
}
\]
\end{definition}

When $\cO$ is simply all of $\Set^G$, the category of polynomials with
exponent in $\cO$ is precisely the category of Tambara functors.  When
$\cO$ is a proper subcategory, the fact that construction of
$\cP^G_{\cO}$ actually forms a category is not immediately obvious; we
now explain the composition, giving a digest form of~\cite[Section
  2.2]{BHOTamb}.  To describe the composition, it is useful to isolate
a generating collection of morphisms.

\begin{definition}\label{defn:TNR}
Let $f\colon S\to T$ be a map of finite $G$-sets. Then let
\begin{align*}
R_f&:=[T\xleftarrow{f} S\xrightarrow{=} S\xrightarrow{=}S] \\
N_f&:=[S\xleftarrow{=} S\xrightarrow{f} T\xrightarrow{=}T] \\
T_f&:=[S\xleftarrow{=} S\xrightarrow{=} S\xrightarrow{f} T]
\end{align*}
\end{definition}

Any polynomial with exponent in $\Set^G$ can be written as a composite
of these basic generators: 
\[
T_h\circ N_g\circ R_f=[S\xleftarrow{f} U\xrightarrow{g} V\xrightarrow{h} T].
\]
(This decomposition motivated Tambara's initial name for these as TNR
functors~\cite{Tambara}.)

Furthermore, we have a series of commutation relations that allow us
to turn any composite into one of this form.

\begin{proposition}
$R$ gives a contravariant functor from $\Set^G$ into $\cP^G_{\cO}$. $N$ gives a covariant functor from $\cO$ to $\cP^G_{\cO}$, and $T$ give covariant functor from $\Set^G$.
\end{proposition}

\begin{proposition}
If we have a pullback diagram of finite $G$-sets
\[
\xymatrix{
{S'}\ar[r]^{f'}\ar[d]_{g'} & {T'}\ar[d]^{g} \\
{S}\ar[r]_{f} & {T,}
}
\]
then we have
\[
R_g\circ N_f=N_{f'}\circ R_{g'}\text{ and }R_g\circ T_f=T_{f'}\circ R_{g'}.
\]
\end{proposition}

Here is where the pullback stability of $\cO$ first appears: since $\cO$ is assumed to be pullback stable, $f$ being in $\cO$ implies that its pullback $f'$ also is.

The interchange of $N$ and $T$ is trickier. Recall that if $f\colon S\to T$ is a map of finite $G$-sets, then the pullback functor
\[
f^\ast\colon\Set^G_{\downarrow T}\to\Set^G_{\downarrow S}
\]
has a right adjoint: the dependent product $\prod_{f}$. Heuristically, this is the product of the fibers over the preimages of $f$, though we will never need the actual form.

\begin{definition}
An exponential diagram in $\Set^G$ is a diagram (isomorphic to one) of the form
\[
\xymatrix{
{S}\ar[d]_{h} & {A}\ar[l]_g & {S\times_{T}\prod_h A}\ar[l]_-{f'}\ar[d]^{g'} \\
{T} & & {\prod_h A.}\ar[ll]^{h'}
}
\]
\end{definition}

\begin{proposition}\label{prop:NormofTransfer}
If we have an exponential diagram
\[
\xymatrix{
{S}\ar[d]_{g} & {A}\ar[l]_h & {S\times_{T}\prod_g A}\ar[l]_-{f'}\ar[d]^{g'} \\
{T} & & {\prod_g A,}\ar[ll]^{h'}
}
\]
then 
\[
N_g\circ T_h=T_{h'}\circ N_{g'}\circ R_{f'}.
\]
\end{proposition}

Again, pullback stability of $\cO$ is all that is needed to guarantee that the subscripts of any maps $N$ are all drawn from $\cO$. 

\begin{proposition}[{\cite[Proposition 2.12]{BHOTamb}}]
Disjoint union of finite $G$-sets is the categorical product in $\cP^G_{\cO}$.
\end{proposition}

\begin{definition}[{\cite[Definition 4.1]{BHOTamb}}]\label{def:OTamb}
An $\cO$-semi-Tambara functor is a product preserving functor
$\cP^G_{\cO}\to\Set$. An $\cO$-Tambara functor is an
$\cO$-semi-Tambara functor that whose value at each finite $G$-set is
an abelian group. 
\end{definition}

Following Tambara's original argument, the category of $\cO$-Tambara
functors as described in Definition~\ref{def:OTamb} is equivalent to
the category of additive functors to abelian groups from the
``additive completion'' of $\cP^G_{\cO}$, i.e., the category obtained
by group-completing the hom monoids of $\cP^G_{\cO}$.  We will use
this equivalence in the remainder of the paper without comment.

The role that the integers plays for abelian groups and commutative rings is played by the Burnside Mackey functor in Mackey and Tambara functors.

\begin{definition}
The Burnside Mackey functor \(\mA\) assigns to each finite \(G\)-set \(T\) the Grothendieck group of the comma category of finite \(G\)-sets over \(T\). The transfer map associated to \(f\colon T\to T'\) is just the composition:
\[
[S\xrightarrow{g} T]\mapsto [S\xrightarrow{f\circ g} T'],
\] 
while the restriction is given by the pullback. The norms are given by the dependent product \(\Pi_{f}\).
\end{definition}

Unpacking this for orbits, we have that 
\(
\mA(G/H)
\)
is group completion of the category of finite \(H\)-sets (with respect to disjoint union). The restriction map is the obvious forgetful map, and the transfer is given by induction. Here the norm is given by coinduction. In particular, the group \(\mA(G/H)\) is the free abelian group generated by the isomorphism classes of orbits \(H/K\), and the multiplication is determined by Frobenius reciprocity:
\[
[H/K]\cdot[H/L]=Tr_{K}^{H} Res_{K}^{H}[H/L]=Tr_{L}^{H}Res_{L}^{H}[H/K].
\]

In later sections, we will restrict attention to \(C_{2}\), so we give a shorter name to the free orbit.
\begin{notation}
Let \(t\in \mA(C_{2}/C_{2})\) be the element \([C_{2}]\).
\end{notation}

\section{Free \texorpdfstring{$\cO$}{O}-Tambara functors}\label{sec:freeab}

In this section, we give explicit formulas for the free Tambara
functors associated to an indexing system $\cO$.  Here, we work with
an arbitrary finite group $G$.  The formulas established here will
subsequently provide key building blocks for deducing an explicit
expression for the right adjoint.

For any finite $G$-set $T$, we have a Tambara functor 
\[
\mA^{\cO}[x_T]\colon=\cP^{G}_{\cO}(T,-) 
\]
co-representing the functor $\m{R}\mapsto\m{R}(T)$ via the Yoneda Lemma:
\[
\OTamb(\mA^{\cO}[x_T],\mR)\cong \mR(T).
\]
In all that follows, when we say that an element in \(\mR\) is ``adjoint to a map from a free Tambara functor'', we mean via this Yoneda isomorphism.

We now unpack how naturality and the structure maps on $\m{R}$ give a dual structure on
$T\mapsto \mA^{\cO}[x_T]$.

Since for any Tambara functor $\m{R}$ and for any finite $G$-set $T$,
the value $\m{R}(T)$ is naturally a commutative ring, $\mA^{\cO}[x_T]$
has a canonical co-ring structure.  We begin by identifying several
distinguished elements in these \(\cO\)-Tambara functors. 

\begin{definition}
If \(T\) is a finite \(G\)-set, then let \(x_{T}\in \mA^{\cO}[x_{T}](T)\) denote the element represented by the polynomial
\[
[T\xleftarrow{1} T\xrightarrow{1} T\xrightarrow{1} T].
\]

If \(T=T_{0}\amalg T_{1}\), then for \(i=0,1\), let \(x_{T_{i}}\in
\mA^{\cO}[x_{T}](T_{i})\) be the element: 
\[
[T\leftarrow T_{i}\xrightarrow{1} T_{i}\xrightarrow{1} T_{i}].
\]
\end{definition}

When \(T=T'\amalg T'\), then we will use different lower-case Roman
letters to denote the two copies of \(x_{T}\).

\begin{remark}
The elements \(x_{T_{i}}\) are the restrictions of \(x_{T}\) to
\(T_{i}\) along the natural inclusions. These elements also freely
generate the Tambara functor; for any \(\cO\)-Tambara functor
\(\mR\), when $T = T_0 \amalg T_1$ we have a canonical isomorphism
\[
\mR(T)\cong \mR(T_{0})\times\mR(T_{1}).
\]
The elements \(x_{T_{i}}\) represent the inclusions of the summands.
\end{remark}

Having defined elements \(x_{T}\), we can now consider any combination of transfers, norms, or restrictions on them. In particular, we have any polynomial in the ordinary sense.

\begin{definition}\label{def:coRing}
Let the {\defemph{co-addition}} and {\defemph{co-multiplication}}
maps 
\[
\Delta_{+}\colon \mA^{\cO}[x_T]\to \mA^{\cO}[x_T,y_T]\text{ and }\Delta_{\times}\colon \mA^{\cO}[x_T]\to\mA^{\cO}[x_T,y_T],
\]
be the maps adjoint to the elements 
\begin{align*}
x_{T}+y_{T}&= [T\amalg T\xleftarrow{1} T\amalg T\xrightarrow{1} T\amalg T\xrightarrow{\nabla} T]\\
x_{T}\cdot y_{T}&= [T\amalg T\xleftarrow{1} T\amalg T\xrightarrow{\nabla} T\xrightarrow{1} T],
\end{align*}
where \(\nabla\) is the fold map, in $\mA^{\cO}[x_T,y_T](T)$. 

Let the {\defemph{co-additive and co-multiplicative counits}}
\[
\epsilon_+\colon\mA^{\cO}[x_T]\to\mA\text{ and }\epsilon_{\times}\colon\mA^{\cO}[x_T]\to\mA
\] 
be the maps adjoint to $0$ and $1$ respectively. 

Finally, let the {\defemph{co-additive co-inversion}}
\[
c\colon\mA^{\cO}[x_T]\to\mA^{\cO}[x_T]
\]
be the map adjoint to $-x_T$. 
\end{definition}

Summarizing how these operations fit together, we have the following proposition.

\begin{proposition}\label{prop:coRing}
The maps $\Delta_{+}$, $\epsilon_+$, $c$, $\Delta_{\times}$, and
$\epsilon_{\times}$ make $\mA^{\cO}[x_T]$ into a co-commutative
co-ring object in $\cO$-Tambara functors. 
\end{proposition}

\begin{remark}
We want to stress that in general, $\mA^{\cO}[x_T]$ is not flat over
$\mA$, so a reader versed in the Hopf-schools will at this point be
sorely disappointed. 
\end{remark}

More generally, we have a full co-Tambara functor structure on
$\mA^{\cO}[x_T]$, with these co-ring maps being subsumed by
co-transfer and co-norm maps and being linked by co-restriction maps. 

\begin{definition}\label{def:CoTNR}
Let $f\colon T\to S$ be a map of finite $G$-sets.   
Then let the {\defemph{co-restriction}} associated to $f$ 
\[
\Delta_{R_f}\colon \mA^{\cO}[x_T]\to\mA^{\cO}[x_S]
\]
to be the map adjoint to
\[
R_f(x_S)= [S\xleftarrow{f} T\xrightarrow{1} T\xrightarrow{1} T]\in\mA^{\cO}[x_S](T).
\]

Let the {\defemph{co-transfer}} associated to $f$
\[
\Delta_{T_f}\colon\mA^{\cO}[x_S]\to \mA^{\cO}[x_T]
\]
be the map adjoint to 
\[
T_f(x_T)=[T\xleftarrow{1} T\xrightarrow{1} T\xleftarrow{f}S]\in \mA^{\cO}[x_T](S). 
\]

Let the {\defemph{co-norm}} associated to $f$
\[
\Delta_{N_f}\colon\mA^{\cO}[x_S]\to\mA^{\cO}[x_T]
\]
be the map adjoint to 
\[
N_f(x_T)=[T\xleftarrow{1} T\xrightarrow{f} S\xrightarrow{1} S]\in\mA^{\cO}[x_T](S). 
\]
\end{definition}

The co-restriction map is a map of co-rings, since the dual statement
is universally true.  Since the transfer (norm) associated to the fold
map is the addition (multiplication), the maps $\Delta_{T_f}$ and
$\Delta_{N_f}$ subsume the co-addition and co-multiplication maps
described before. Additionally, naturality shows that these
co-transfer (co-norm) maps compose in the expected way. 

Finally, we note that the universal formulae expressing the norm of a transfer (Proposition~\ref{prop:NormofTransfer}) gives an identity
\[
\Delta_{T_{h}}\circ\Delta_{N_{g}}=\Delta_{R_{f'}}\circ\Delta_{N_{g'}}\circ \Delta_{T_{h'}},
\]
where \(f'\), \(g'\), and \(h'\) are as in Proposition~\ref{prop:NormofTransfer}. We will not need this relation, however.
These observations and
the Yoneda lemma thus imply the following result.

\begin{proposition}\label{prop:CoTambaraonFrees}
The assignment 
\[
\mA^{\cO}[x_{(-)}]:=T\mapsto \mA^{\cO}[x_T]
\]
together with the structure maps given by Definitions~\ref{def:coRing} and \ref{def:CoTNR} gives a co-$\cO$-Tambara functor object in $\cO$-Tambara functors.

For a fixed $\cO$-Tambara functor $\m{R}$, the $\cO$-Tambara functor 
\[
T\mapsto \OTamb\big(\mA^{\cO}[x_T],\m{R}\big)
\]
is canonically isomorphic to $\m{R}$.
\end{proposition}

In the case when $T$ is a trivial $G$-set, we simply obtain a
polynomial ring; this is easy to see as follows.

\begin{definition}
Let \(-\otimes\mA\) be the left adjoint to the functor which evaluates a Mackey functor on the finite \(G\)-set \(G/G\). 
\end{definition}

\begin{lemma}
If \(M\) is an abelian group, then \(M\otimes\mA\) is the Mackey functor
\[
T\mapsto M\otimes \mA(T).
\]
The structure maps are determined by those of \(\mA\).
\end{lemma}
\begin{proof}
By definition, the left adjoint applied to \(\Z\) gives \(\mA\). Since left adjoints commute with colimits, the result follows from choosing a free resolution of \(M\) and viewing it as a reflexive coequalizer.
\end{proof}

\begin{proposition}
The functor \(-\otimes\mA\) is a strong symmetric monoidal functor.
\end{proposition}
\begin{proof}
The symmetric monoidal product on Mackey functors is a closed symmetric monoidal structure. The internal \(\Hom\) is closely connected to the ordinary \(\Hom\) in Mackey functors:
\[
\Hom(\mM,\m{N})\cong \m{\Hom}(\mM,\m{N})(G/G).
\]

We then have
\begin{align*}
\Hom\big((M\otimes\mA)\Box (N\otimes\mA), \m{B}\big)&\cong
\Hom\big(M\otimes\mA, \m{\Hom}(N\otimes\mA,\m{B})\big) \\
&\cong \Hom\big(M,\m{\Hom}(N\otimes\mA,\m{B})(G/G)\big) \\
&\cong \Hom\big(M,\Hom(N\otimes\mA,\mB)\big) \\
&\cong \Hom\big(M,\Hom(N,\m{B}(G/G))\big) \\
&\cong \Hom\big(M\otimes N, \m{B}(G/G)\big) \\
&\cong \Hom\big((M\otimes N)\otimes \mA,\m{B}\big). \qedhere
\end{align*}
\end{proof}

\begin{corollary}\label{cor:FreeGreenFixed}
If \(T\) is a finite set with trivial \(G\)-action, then
\[
\mA^{\cO}[x_{T}]\cong \Z[t\mid t\in T]\otimes \mA.
\]
\end{corollary}

\section{Free \texorpdfstring{$C_{2}$}{C2} Green and Tambara functors}\label{sec:freeC2}

We now specialize to the case where $G = C_2$ and work out in detail
what the descriptions of the previous section look like.  The more
general case of $G = C_p$ is completely analogous.

\begin{notation}
Since there are only two subgroups of $C_2$, Mackey, Green, and Tambara
functors are completely determined by their values on $C_2/C_2$ and
$C_2/e$. We will refer to the value of a Mackey functor at $C_2/H$ as
its {\defemph{value at level $C_2/H$}}. We will refer to elements in
the value at $C_2/C_2$ as {\defemph{fixed}} and to elements in the
value at $C_2/e$ as {\defemph{underlying}}.
\end{notation}

There are only two indexing systems for $C_2$, since there are only 2
subgroups of $C_2$:
\begin{enumerate}
\item $\cO$ the trivial coefficient system, giving Green functors, and
\item $\cO'$ the complete coefficient system, giving Tambara functors.
\end{enumerate}

We begin with the free Green functor.  In the formulae below we adopt
the convention that for a finite $G$-set $T$, a subscript of $T$ on an
element will indicate that this is an element in $\m{R}(T)$.

\subsection{Free \texorpdfstring{\(C_{2}\)}{C2} Green Functors}
Corollary~\ref{cor:FreeGreenFixed} describes the free Green functor on a fixed generator. We need now only determine the free Green functor on an underlying generator.

\begin{lemma}\label{lem:FreeGreenC2}
We have isomorphisms of $C_2$-rings
\[
\mA^{\cO}[x_{C_2}](C_2/e)\cong \Z[x,\bar{x}],
\]
where \(\bar{x}\) is the Weyl conjugate of \(x\).

We have an isomorphism of rings
\begin{multline*}
\mA^{\cO}[x_{C_2}](C_2/C_2)\cong \\
\mA(C_2/C_2)\big[\{t_{i,j} | 0\leq i,  j\}\big]/\big(t_{i,j}=t_{j,i}, t_{0,0}=t, t_{i,j}\cdot t_{k,\ell}=t_{i+k,j+\ell}+t_{i+\ell, j+k}\big).
\end{multline*}
The restriction is determined by
\[
res_{e}^{C_2}(t_{i,j})=x^i\bar{x}^j+x^j\bar{x}^i,
\]
while the transfer of is given by
\[
tr_e^{C_2}(x^i\bar{x}^j)=t_{i,j}.
\]
\end{lemma}
Recall that the element \(t\) represents the free orbit \([C_{2}]\) in \(\mA(C_{2}/C_{2})\).
\begin{proof}
The first part is simply the statement that when we restrict to \(C_{2}\)-rings, the free Green functor on a class \(x_{C_{2}}\) is the free commutative algebra in \(C_{2}\)-rings on a class \(x\). We can also express the generators easily in the language of polynomials from Definition~\ref{defn:Polynomials}:
\[
x=\big[C_{2}\xleftarrow{1}C_{2}\xrightarrow{1} C_{2}\xrightarrow{1} C_{2}\big]\text{ and }\bar{x}=\big[C_{2}\xleftarrow{1} C_{2}\xrightarrow{\gamma} C_{2} \xrightarrow{1} C_{2}\big].
\]
The equivalence relation on polynomials shows that any span like this will represent either \(x\) or \(\bar{x}\): we need only count the number of occurrences of \(\gamma\).

The elements of \(\mA^{\cO}[x_{C_{2}}](\ast)\) are polynomials
\begin{equation}\label{eqn:Polynomial}
C_{2}\xleftarrow{f} S_{0}\xrightarrow{g} S_{1}\xrightarrow{h} \ast.
\end{equation}
Since we are mapping to \(C_{2}\), the map \(f\) is of the form
\[
C_{2}\times f_{0}\colon C_{2}\times S'\to C_{2}\times \ast\cong C_{2},
\]
where \(S'\) has a trivial \(C_{2}\)-action. Since we are considering only Green functors, we know that the map \(g\) must preserve the stabilizers of points, and hence \(S_{1}\) must be free. This forces \(h\) to be the composite of iterated fold maps with the canonical map \(C_{2}\to\ast\), and since the fold map corresponds to the addition and the map \(C_{2}\to\ast\) to the transfer, we understand exactly what the map \(T_{h}\) gives.

The map \(g\) is more interesting. Our analysis of \(T_{h}\) allows us to work one orbit in \(S_{1}\) at a time, so without loss of generality, our polynomial is of the form
\[
C_{2}\xleftarrow{f} \coprod_{S'} C_{2}\xrightarrow{g} C_{2}\to\ast.
\]
On each summand, the map \(g\) is either the identity or multiplication by \(\gamma\), a generator of \(C_{2}\). Folding these all together gives the product, and using the identifications of \(x\) and \(\bar{x}\) above, we see that this polynomial represents
\[
tr_{e}^{C_{2}}(x^{i}\bar{x}^{j}),
\]
where \(i\) is the number of summands on which \(g\) is the identity and where \(j\) is the number on which it is multiplication by \(\gamma\). Any other representative of this polynomial will either be a rearranging of the summands (which clearly does not change the result) or composing each map with multiplication by \(\gamma\). This switches the roles of \(i\) and \(j\), giving the symmetry condition.

The final multiplicative relations are just Frobenius reciprocity.
\end{proof}

In words, everything at level $C_2/C_2$ besides the multiplicative unit is in the image of the transfer, and all of the multiplicative relations are determined by the Frobenius relation
\[
a\cdot tr_e^{C_2}(b)=tr_e^{C_2}\big(res_e^{C_2}(a)\cdot b\big).
\]

\begin{remark}
The kernel of the restriction map is the ideal induced from the augmentation ideal in the Burnside ring. If we consider the image of the restriction, then it is exactly the image of the trace map $x\mapsto x+\bar{x}$. In particular, we do not see the norm class $x\bar{x}$ in the image. In some sense, the free Tambara functor on an underlying generator below exactly fixes this problem.
\end{remark}

\subsection{Free \texorpdfstring{\(C_{2}\)}{C2} Tambara functors}\label{ssec:FreeTambara}
The free Tambara functor on one fixed generator can be built out of
the free Green functor by formally adding all of the missing norms. In
this case, the situation is greatly simplified: the underlying ring is
polynomial: 
\[
\mA^{\cO'}[x_\ast](C_2/e)=\Z[x].
\]
The norm is multiplicative, so the norm of any monomial of the form \(ax^{n}\) with \(a\in\Z\) is determined by the norm of \(x\) and by the norm of \(a\). The latter is given by the formulae for norms in the Burnside Tambara functor. For more general polynomials in \(x\), we use a universal formula. 

\begin{lemma}[{\cite{Tambara}, \cite[Theorem 3.5]{MazurArxiv}}]\label{lem:NormofSum}
Let $\m{R}$ be a $C_2$-Tambara functor, and let $a,b\in\m{R}(C_2/e)$. Then in $\m{R}(C_2/C_2)$ we have
\[
N_e^{C_2}(a+b)=N_e^{C_2}(a)+N_e^{C_2}(b)+tr_e^{C_2}(a\bar{b}).
\]
\end{lemma}

Since the Weyl action in our case is trivial, in fact a simpler
formula holds.

\begin{proposition}\label{prop:normofsum}
Let $p(x)\in\Z[x]$ be a polynomial of degree at most $(k-1)$. Then in
the free Tambara functor on one fixed generator $\mA^{\cO'}\![x_\ast]$,
we have  
\[
N_e^{C_2}\big(ax^k+p(x)\big)=N_e^{C_2}(a) N_e^{C_2}(x)^k+N_e^{C_2}\big(p(x)\big)+t\cdot \big(ax^k\cdot p(x)\big).
\]
\end{proposition}

We record all of this explicitly in the following lemma.

\begin{lemma}\label{lem:FreeTambaraonFixed}
We have an isomorphism of rings
\[
\mA^{\cO'}\![x_{\ast}](\ast)\cong \mA(\ast)[x_{\ast},n_{\ast}]/\big(t(x_{\ast}^{2}-n_{\ast})\big).
\]
We have an isomorphism of \(C_{2}\)-rings
\[
\mA^{\cO'}\![x_{\ast}](C_{2}/e)\cong \Z[x].
\]
The restriction map sends \(x_{\ast}\) to \(x\) and \(n_{\ast}\) to \(x^{2}\). The transfer is multiplication by \(t\). The norm is given by Proposition~\ref{prop:normofsum}, where \(N_{e}^{C_{2}}(x)=n_{\ast}\).
\end{lemma}

\begin{proof}
The statement about the underlying ring is immediate. Also immediate from the Yoneda lemma is that the class \(x_{\ast}\) is represented by the polynomial
\[
x_{\ast}=\ast\leftarrow \ast\to\ast\to\ast.
\]
This restricts to the class
\[
x=\ast\leftarrow C_{2}\xrightarrow{1} C_{2}\xrightarrow{1} C_{2}.
\]

Let \(n_{\ast}\) be the polynomial
\[
n_{\ast}=\ast\leftarrow C_{2}\to\ast\to\ast.
\]
When we restrict this to \(C_{2}\), we get
\[
\ast\leftarrow{C_{2}\amalg C_{2}}\xrightarrow{\nabla} C_{2}\xrightarrow{1} C_{2},
\]
which is the square of the class \(x\) (here we use that we are mapping to \(\ast\) on the left to show that the choice of ways to write \(C_{2}\times C_{2}\) as \(C_{2}\amalg C_{2}\) do not matter). The relation \(t(x_{\ast}^{2}-n_{\ast})\) actually follows from this:
\[
t(x_{\ast}^{2}-n_{\ast})=tr_{e}^{C_{2}}\big(res_{e}^{C_{2}}(x^{2}_{\ast}-n_{\ast})\big)=tr_{e}^{C_{2}}\big(0\big).
\]

We again argue the rest via polynomials. Elements of the free Tambara functors are polynomials
\[
\ast\xleftarrow{f} S_{0}\xrightarrow{g} S_{1}\xrightarrow{h}\ast.
\]
Consider an orbit decomposition of \(S_{1}\):
\[
S_{1}=\left(\coprod_{S_{1}'}\ast\right)\amalg \left(\coprod_{S_{1}''}C_{2}\right).
\]
On the summand indexed by \(S_{1}'\), the map \(T_{h}\) gives the ordinary addition, while on the summand indexed by \(S_{1}''\), it gives the transfer. It therefore again suffices to consider \(S_{1}\) a single orbit.

When \(S_{1}=\ast\), then we consider the decomposition of \(S_{0}\) into orbits:
\[
S_{0}=\left(\coprod_{k=1}^{i}C_{2}\right)\amalg \left(\coprod_{\ell=1}^{j}\ast\right),
\]
and our polynomial represents
\[
n_{\ast}^{i}x_{\ast}^{j}.
\]
Moreover, these are all distinct and linearly independent.

When \(S_{1}=C_{2}\), then we must have that \(S_{0}\) is a union of copies of \(C_{2}\), and the map is just composites of the fold and multiplication by \(\gamma\) maps. Here, since the last term in the polynomial is also \(\ast\), we have an isomorphism
\[
(\ast\leftarrow C_{2}\xrightarrow{1} C_{2}\to \ast)\cong (\ast\leftarrow C_{2}\xrightarrow{\gamma} C_{2}\to\ast).
\]
In particular, we may assume that the map \(S_{0}\to C_{2}\) is just an iterated fold map:
\[
\coprod_{k=1}^{i}C_{2}\to C_{2}.
\] 
The fold map gives the multiplication, so this polynomial is 
\[
tr_{e}^{C_{2}}(x^{i}).
\]
Since 
\[
x^{i}=res_{e}^{C_{2}} x_{\ast}^{i},
\]
the \(\mA\)-module structure allows us to rewrite this polynomial as
\[
t\cdot x_{\ast}^{i}.
\]
This gives the result.
\end{proof}

Finally, we describe the structure of the free Tambara functor on a
free generator.

\begin{lemma}\label{lem:FreeTambaraonC2}
The ring underlying the free Tambara functor \(\mA^{\cO'}\![x_{C_{2}}]\) is \(\Z[x,\bar{x}]\).

The fixed ring for \(\mA^{\cO'}\![x_{C_{2}}]\) is
\begin{multline*}
\mA^{\cO}[x_{C_{2}}][n]/ \big(t_{i,j}\cdot n-t_{i+1,j+1}\big)\\
\cong \Z[t]/(t^2-2t)\big[t_i \mid 0\leq i\big][n]/\big(t_0-t, t_i\cdot t_j-(t_{i+j}+n^j t_{i-j}), i\leq j\big).
\end{multline*}

The restriction map is that of \(\mA^{\cO}[x_{C_{2}}]\), together with 
\[
res_{e}^{C_{2}}n=x\bar{x}.
\]
The norms are determined by 
\[
N_{e}^{C_{2}}(x)=N_{e}^{C_{2}}(\bar{x})=n.
\]
\end{lemma}
\begin{proof}
The underlying ring is the same as the ring underlying the free Green functor on a class \(x_{C_{2}}\), so this part is clear.

For the fixed ring, we again argue via polynomials. The class \(n\) is represented by
\[
n= \big(C_{2}\xleftarrow{1} C_{2}\to \ast\to\ast\big),
\]
which restricts to 
\[
\big(C_{2}\xleftarrow{\nabla} C_{2}\amalg C_{2}\xrightarrow{1\amalg \gamma} C_{2}\to C_{2}\big)=x\bar{x}
\]
By definition, the class \(n\) is the norm of \(x\) and \(\bar{x}\).

We now mirror the computations of Lemma~\ref{lem:FreeGreenC2}. Since the source is \(C_{2}\), any polynomial can be written as
\[
C_{2}\xleftarrow{C_{2}\times q} C_{2}\times S_{0}\xrightarrow{g} S_{1}\xrightarrow{h}\ast,
\]
where \(q\) is the canonical quotient. Since we are considering the free Tambara functor, we can have any isotropy for the points of \(S_{1}\), and again, the map \(T_{h}\) gives a sum of transfers, depending on the stabilizers of points in \(S_{1}\). We therefore again are reduced to checking on individual orbits.

If \(S_{1}=\ast\), then by definition, the polynomial is just \(n^{|S_{0}|}\). If \(S_{1}=C_{2}\), then the polynomial is exactly one of the ones considered in Lemma~\ref{lem:FreeGreenC2}. In particular, any class is a linear combination of an element of \(\mA^{\cO}[x_{C_{2}}](C_{2}/C_{2})\) and an element of  \(\Z[n]\), and these two rings sit as subrings. The products between them are given by Frobenius reciprocity, from which the stated result follows.
\end{proof}

\section{The operadic right adjoint}\label{sec:main}

Using our abstract analysis of the free functor, we can easily give an
abstract description of the right adjoint.  Let $\cO\subset \cO'$ be
indexing systems, and let 
\[
{\OpRes}\colon\cO'\mhyphen\Tamb\to\OTamb
\]
be the restriction functor. 

\begin{theorem}\label{prop:ValuesofRightAdjoint}
We have a natural isomorphism for any finite $G$-set $T$
\[
F_{\cO}(\cO'\!,\m{R})(T)\cong \OTamb\Big({\OpRes}\big(\mA^{\cO'}\![x_T]\big),\m{R}\Big).
\]
The structure maps in $\m{R}$ are dual to ${\OpRes}$ of the co-structure maps in $\mA^{\cO'}\![x_{(-)}]$.
\end{theorem}

\begin{proof}
This is the composite of the isomorphism in
Proposition~\ref{prop:CoTambaraonFrees} and the adjunction: 
\[
F_{\cO}(\cO'\!,\m{R})(T)\cong
\cO'\mhyphen\Tamb\big(\mA^{\cO}[x_T],F_{\cO}(\cO',\m{R})\big)\cong\OTamb\Big({\OpRes}\big(\mA^{\cO'}\![x_T]\big),\m{R}\Big). 
\]
The statement about structure maps is the second part of
Proposition~\ref{prop:CoTambaraonFrees}. 
\end{proof}

To illustrate how this works, we work out explicitly what happens in
the case $G=C_2$, the cyclic group of order $2$, using our analysis
from above.  The answer is again completely analogous for $G = C_p$, $p$ an
odd prime.  However, there is one point where the formulae are
significantly prettier for $p=2$ (indicated below), so we focus on
this case.

We describe the answer in the following two theorems.  The first
characterizes the values of the right adjoint in terms of certain
pullback diagrams of multiplicative monoids.  The second omnibus
theorem describes formulas for the various operations that are
constituents of the Tambara functor structure.
We prove these results via a series of propositions and corollaries
establishing various parts.  For expositional convenience, we indicate
in the statements the supporting results which prove the individual
pieces.

\begin{theorem}[Corollaries~\ref{cor:RightAdjointSet}
  and~\ref{cor:MultiplicativeStructure}]
Let $\m{R}$ be a Green functor.  Then we have pullback squares of multiplicative monoids
\[
\xymatrix{
{F_{\cO}\big(\cO'\!,\m{R}\big)(C_2/C_2)}\ar[r]\ar[d] & {\m{R}(C_2/C_2)}\ar[d]^{\res} \\
{\m{R}(C_2/C_2)}\ar[r]_{(\res)^2} & {\m{R}(C_2/e),}
}
\]
where $\res$ is the restriction and $(\res)^2$ is the composite of the squaring map with the restriction, and
\[
\xymatrix{
{F_{\cO}\big(\cO'\!,\m{R}\big)(C_2/e)}\ar[r]\ar[d] & {\m{R}(C_2/C_2)}\ar[d]^{\res} \\
{\m{R}(C_2/e)}\ar[r]_{x\mapsto x\cdot\bar{x}} & {\m{R}(C_2/e),}
}
\]
where here $\bar{x}$ denotes the Weyl conjugate of $x$ and $x\mapsto
x\cdot\gamma x$ is the map sending an element to the product over the
group of its Weyl conjugates.  The multiplication in each square is
coordinate-wise.
\end{theorem}

Since the values of \(F_{\cO}\big(\cO'\!,\m{R}\big)(C_{2}/H)\) are pullbacks, we will name elements as ordered pairs, the first element of which will always come from the top right set in the square and the second will come from the bottom left set. 

\begin{theorem}
If \(\mR\) is a commutative \(C_{2}\)-Green functor, then the Tambara structure on \(F_{\cO}\big(\cO'\!,\m{R}\big)\) is given by the following series of results

\begin{enumerate}

\item (Corollary~\ref{cor:FixedAddition}) The addition on $F_{\cO}(\cO'\!,
  \m{R})(C_2/C_2)$ is given by
\[
(n,x)+(n',x')=(n+n'+t\cdot x\cdot x',x+x')
\]
\vspace{5 pt}
\item (Corollary~\ref{cor:UnderlyingAddition}) The addition on $F_{\cO}(\cO'\!,\m{R})(C_2/e)$ is given by
\[
(n,x)+(n',x')=\big(n+n'+\tr(x\cdot\bar{x}'),x+x'\big)
\]
\vspace{5 pt}
\item (Corollary~\ref{cor:RightAdjointNorm}) The norm map 
\[
N_e^{C_2} \colon F_{\cO}(\cO'\!,\m{R})(C_2/e) \to F_{\cO}(\cO'\!,\m{R})(C_2/C_2)
\]
is specified by
\[
(n,x)\mapsto (n^2,n).
\]
\vspace{5 pt}
\item (Corollary~\ref{cor:RightAdjointRestriction}) The restriction 
\[
\res_e^{C_2} \colon F_{\cO}(\cO'\!,\m{R})(C_2/C_2) \to F_{\cO}(\cO'\!,\m{R})(C_2/e)
\]
is given by
\[
(n,x)\mapsto\big(n,\res(x)\big).
\]
\vspace{5 pt}
\item (Corollary~\ref{cor:RightAdjointTransfer}) The transfer 
\[
\tr_e^{C_2} \colon F_{\cO}(\cO'\!,\m{R})(C_2/e) \to F_{\cO}(\cO'\!,\m{R})(C_2/C_2)
\]
is specified by
\[
(n,x)\mapsto \big(2n+\tr(x^2),\tr(x)\big).
\]
\end{enumerate}
\end{theorem}

\begin{remark}
The formula for the norm perhaps best exemplifies the weirdness of the right adjoint. Here, the coordinate actually coming from \(\mR(C_{2}/e)\) (the element \(x\) above) is completely ignored in forming the norm.
\end{remark}

To prove this theorem, we repeatedly apply
Theorem~\ref{prop:ValuesofRightAdjoint}, which tells us that we need
only determine the restriction to Green functors of the two free Tambara functors described in Section~\ref{ssec:FreeTambara}: 
\[
{\OpRes}\mA^{\cO'}\![x_{\ast}]\text{ and }{\OpRes}\mA^{\cO'}\![x_{C_2}],
\]
together with the co-Tambara functor structure maps
\[
\Delta_{+}, \Delta_{\times}, \Delta_{r}, \Delta_{t}, \text{ and }\Delta_n,
\]
where $r$, $t$, and $n$ are the restriction, transfer, and norm
associated to $C_2\to\ast$.  To this end, we apply the work of Section~\ref{sec:freeC2}.

Since the free Green functor on two fixed generators has fixed points
\[
\mA(C_{2}/C_{2})[x,n]
\]
and underlying 
\[
\Z[x,n],
\]
we see the the restriction of \(\mA^{\cO'}\![x_{\ast}]\) to a Green
functor is just the quotient of the free Green functor on two
generators by the Green ideal generated by the class \(x^{2}-n\). This
gives the following proposition. 

\begin{proposition}\label{prop:FixedATamb}
We have a pushout diagram of commutative Green functors
\begin{center}
\begin{tikzcd}
\mA^{\cO}[y_{C_2}]
	\arrow[r, "R_n"]
	\arrow[d, "R_x"']
	&
\mA^{\cO}[n_{\ast}]
	\arrow[d]
	\\
\mA^{\cO}[x_{\ast}]
	\arrow[r]
	&
i^{\cO'}_{\cO} \mA^{\cO'}\![x_{\ast}],
\end{tikzcd}
\end{center}
where $R_n$ is the map adjoint to the element
\[
res_e^{C_2} n_\ast\in \mA^{\cO}[n_\ast](C_2/e)
\]
and where $R_x$ is the map adjoint to the element
\[
res_e^{C_2}(x_{\ast}^2)=\big(res_e^{C_2}x_{\ast}\big)^2\in \mA^{\cO}[x_{\ast}](C_2/e).
\]
In the underlying Tambara functor structure, we have
\[
n_\ast=N_e^{C_2} \big(res_e^{C_2} x_\ast\big).
\]
\end{proposition}

The analysis for the free Tambara functor on an underlying generator is similar. We again observe that although the fixed ring is quite messy, we have a simple ``generators and relations'' way to interpret this Green functor. We have a map of Green functors 
\[
\mA^{\cO}[x_{C_{2}},n_{\ast}]\to i^{\cO'}_{\cO}\mA^{\cO'}\![x_{C_{2}}]
\]
which is a surjection upon evaluating at any finite \(C_{2}\)-set. The complicated relations we see in the fixed ring are again in the image of the transfer and arise from the more basic formula
\begin{equation}\label{eqn:ResofNorm}
res_{e}^{C_{2}}n=x\bar{x},
\end{equation}
which is just the multiplicative double coset formula expressing the restriction of a norm. This observation gives the following proposition.

\begin{proposition}\label{prop:UnderlyingATamb}
We have a pushout diagram of commutative Green functors
\begin{center}
\begin{tikzcd}
\mA^{\cO}[y_{C_2}]
	\arrow[r, "R"]
	\arrow[d, "N"']
	&
\mA^{\cO}[n_{\ast}]
	\arrow[d]
	\\
\mA^{\cO}[x_{C_2}]
	\arrow[r]
	&
i^{\cO'}_{\cO} \mA^{\cO'}\![x_{C_2}],
\end{tikzcd}
\end{center}

where $R$ is the map adjoint to the element 
\[
res_{e}^{C_2} n_\ast\in \mA^{\cO}[n_{\ast}](C_2/e)
\]
and where $N$ is the map adjoint to the element
\[
x_{C_2}\cdot \bar{x}_{C_2}\in \mA^{\cO}[x_{C_2}](C_2/e),
\]
where $\bar{x}_{C_2}$ is the Weyl conjugate. In the underlying Tambara functor structure, we have
\[
n_\ast=N_{e}^{C_2} x_{C_2}.
\]
\end{proposition}

\begin{corollary}\label{cor:RightAdjointSet}
Let $\m{R}$ be a commutative Green functor. Then we have two pullback squares of sets
\begin{center}
\begin{tikzcd}
F_{\cO}\big(\cO'\!,\m{R}\big)(C_2/e)
	\arrow[r]
	\arrow[d]
	&
\m{R}(C_2/C_2)
	\arrow[d, "res_e^{C_2}"]
	\\
\m{R}(C_2/e)
	\arrow[r, "x\mapsto x\cdot \bar{x}"']
	&
\m{R}(C_2/e),
\end{tikzcd}
and
\begin{tikzcd}
F_{\cO}\big(\cO'\!,\m{R}\big)(C_2/C_2)
	\arrow[r]
	\arrow[d]
	&
\m{R}(C_2/C_2)
	\arrow[d, "res_e^{C_2}"]
	\\
\m{R}(C_2/C_2)
	\arrow[r, "x\mapsto (res_e^{C_2}x)^2"']
	&
\m{R}(C_2/e).
\end{tikzcd}
\end{center}
In particular for either subgroup $H\subset C_2$, the elements of $F_{\cO}\big(\cO'\!,\m{R}\big)(C_2/H)$ are certain ordered pairs
\[
(n,x_{C_2/H})\in \m{R}(C_2/C_2)\times\m{R}(C_2/H).
\]
\end{corollary}

\begin{remark}
In both pushouts describing the restriction of the free Tambara functor, we have Green functor generators called \(n_{\ast}\). The apparent collision in notation is chosen to reflect that in both cases, this class is the norm of the underlying algebra generator, which is either \(res_{e}^{C_{2}}x_{\ast}\) or \(x_{C_{2}}\) (and its Weyl conjugate). We hope that this overloading of the notation will help reinforce the common role played by both \(n_{\ast}\).
\end{remark}

The comultiplication maps are easily determined in the pullback description: all of the structure maps in the pushouts for Propositions~\ref{prop:UnderlyingATamb} and \ref{prop:FixedATamb} commute with products.

\begin{proposition}\label{prop:CpCoMultiplication}
The co-multiplication map respects the pushout decompositions of Propositions~\ref{prop:UnderlyingATamb} and \ref{prop:FixedATamb}
\end{proposition}
\begin{proof}
The norm map commutes with products, and since all of the structure maps in the pushouts are monomials, they all commute.
\end{proof}

\begin{corollary}\label{cor:MultiplicativeStructure}
The pullback squares of sets in Corollary~\ref{cor:RightAdjointSet} is a pullback in multiplicative monoids: the multiplication maps are done coordinatewise.
\end{corollary}

\begin{proposition}\label{prop:CpCoRestriction}
Recall that the corestriction is the map
\[
\mA^{\cO'}\![x_{C_2}]\to \mA^{\cO'}\![x_{\ast}]
\]
adjoint to $res_{e}^{C_2}(x_{\ast})$. In terms of the pushout decomposition, we have
\[
x_{C_2}\mapsto res_{e}^{C_2}(x_{\ast})\text{ and } n_\ast\mapsto n_{\ast}.
\]
\end{proposition}
\begin{proof}
This is immediate from the description of the elements $x_?$ and $n_\ast$ in Propositions~\ref{prop:UnderlyingATamb} and \ref{prop:FixedATamb}.
\end{proof}

\begin{corollary}\label{cor:RightAdjointRestriction}
In the notation of Corollary~\ref{cor:RightAdjointSet}, the restriction map is given by
\[
res_e^{C_2}\big((n,x_\ast)\big)=(n,res_e^{C_2} x_\ast).
\]
\end{corollary}

\begin{proposition}\label{prop:CpCoNorm}
The co-norm map is the map
\[
\mA^{\cO'}[x_{\ast}]\to \mA^{\cO'}[x_{C_{2}}]
\]
adjoint to $N_e^{C_2} x_{C_2}$. Thus in terms of the pushout generators \(x_{\ast}\) and \(n_{\ast}\) of \(i_{\cO}^{\cO'}\mA^{\cO}[x_{\ast}]\) and \(x_{C_{2}}\) and \(n_{\ast}\) of \(i_{\cO}^{\cO'}\mA^{\cO}[x_{C_{2}}]\), we have
\begin{align*}
x_\ast&\mapsto n_\ast \\
n_\ast&\mapsto n_\ast^2.
\end{align*}
\end{proposition}
\begin{proof}
Only the second formula needs any justification; the first is essentially the definition. Since the co-norm is a map of Tambara functors, we know that we have
\[
n_\ast\mapsto N_e^{C_2} res_e^{C_2} (n_\ast)=N_e^{C_2}(x_{C_2}\bar{x}_{C_2}),
\]
where the last equality is again the multiplicative double coset formula for the restriction of a norm.
The result follows from multiplicativity and Weyl invariance of the norm.
\end{proof}

\begin{corollary}\label{cor:RightAdjointNorm}
In the notation of Corollary~\ref{cor:RightAdjointSet}, the norm map is given by
\[
N_e^{C_2}\big((n,x_{C_2})\big)=(n^2,n).
\]
\end{corollary}

It is much more interesting to determine the co-Mackey functor structure here. Here we have to use the formulae for the norm of a sum, since that shows up in the pushout decomposition. We begin with the coaddition. Here, the coaddition is determined by the maps of Tambara functors
\[
x_{C_2}\mapsto y_{C_2}+z_{C_2}\text{ and }x_{\ast}\mapsto y_{\ast}+z_{\ast}.
\]
For our analysis, we need the underlying map on Green functors given by the pushout squares in Propositions~\ref{prop:UnderlyingATamb} and \ref{prop:FixedATamb}. On the element with same name, the map is obvious. In both of the pushout squares, the elements $n_\ast$ are really norms of the corresponding elements $x_?$, so the image of these is determined by the universal formulae.

\begin{proposition}\label{prop:CpCoTransfer}
The co-transfer map is adjoint to $tr_{e}^{C_2} x_{C_2}$. In terms of the pushout generators given, this gives
\[
x_\ast\mapsto tr_{e}^{C_2} x_{C_2}\text{ and }n_\ast\mapsto 2n_\ast+tr_e^{C_2} (x_{C_2}^2).
\]
\end{proposition}
\begin{proof}
Only the second part of the formula requires any justification. Since the co-transfer is a map of Tambara functors, and since the class $n_\ast$ is the norm of the restriction of $x_\ast$, we conclude 
\[
n_\ast\mapsto N_e^{C_2} res_e^{C_2} tr_e^{C_2} x_{C_2}=N_e^{C_2}\big(x_{C_2}+\bar{x}_{C_2}\big).
\]
The formula then follows from the observation that both $x_{C_2}$ and $\bar{x}_{C_2}$ have the same norm: $n_\ast$ together with Lemma~\ref{lem:NormofSum}.
\end{proof}

\begin{corollary}\label{cor:RightAdjointTransfer}
In the notation of Corollary~\ref{cor:RightAdjointSet}, the transfer is given by
\[
tr_e^{C_2}\big((n,x_{C_2})\big)=\big(2n+tr_e^{C_2}(x_{C_2}^2),tr_e^{C_2} x_{C_2}\big).
\]
\end{corollary}

There are two co-additions: the underlying one (for $C_2/e$) and the fixed one (for $C_2/C_2$). These also follow immediately from Lemma~\ref{lem:NormofSum}.
\begin{proposition}\label{prop:UnderlyingCpCoAddition}
Consider the coaddition map
\[
\mA^{\cO}[x_{C_2}]\to\mA^{\cO}[y_{C_2},z_{C_2}].
\]
For the target, let $n_\ast^y$ and $n_\ast^z$ denote the corresponding norm classes in the pushout description. Then in terms of these pushout generators, the underlying coaddition is given by:
\begin{align*}
x_{C_2}&\mapsto y_{C_2}+z_{C_2} \\
n_\ast&\mapsto n_\ast^{y}+n_\ast^{z}+tr_e^{C_2} (y_{C_2}\bar{z}_{C_2}).
\end{align*}

\end{proposition}

\begin{corollary}\label{cor:UnderlyingAddition}
In the notation of Corollary~\ref{cor:RightAdjointSet}, the underlying addition is given by
\[
(n,x_{C_2})+(n',x_{C_2}')=\big(n+n'+tr_e^{C_2}(x_{C_2}\bar{x}_{C_2}'),x_{C_2}+x_{C_2}'\big).
\]
\end{corollary}

\begin{proposition}\label{prop:FixedCpCoAddition}
Consider the coaddition map
\[
\mA^{\cO}[x_{\ast}]\to\mA^{\cO}[y_{\ast},z_{\ast}].
\]
For the target, let $n_\ast^y$ and $n_\ast^z$ denote the corresponding norm classes in the pushout description. Then in terms of these pushout generators, the underlying coaddition is given by:
\begin{align*}
x_{\ast}&\mapsto y_{\ast}+z_{\ast} \\
n_\ast&\mapsto n_\ast^{y}+n_\ast^{z}+t(y_{\ast}{z}_{\ast}),
\end{align*}
where $t\in\mA(C_2/C_2)$ is the class represented by the $C_2$-set $C_2$.
\end{proposition}
\begin{proof}
The only surprising addition at this point is the element $t$ in the Burnside ring. This simply records a universal formula relating the restriction and transfer to the module structure over the Burnside Mackey functor:
\[
tr_e^{C_2} res_e^{C_2}a=t\cdot a.
\]
\end{proof}

\begin{corollary}\label{cor:FixedAddition}
In the notation of Corollary~\ref{cor:RightAdjointSet}, the fixed addition is given by
\[
(n,x_\ast)+(n',x_\ast')=\big(n+n'+tx_\ast\bar{x}_\ast', x_\ast+x'_\ast \big).
\]
\end{corollary}

\bibliographystyle{amsalpha}

\bibliography{Noo}
\end{document}